\documentclass{amsart}
\usepackage{amsmath,amsthm,amssymb}

\makeatletter
    
    \@addtoreset{equation}{section}
  \makeatother

\newtheorem{definition}{Definition}[section]
\newtheorem{proposition}[definition]{Proposition}
\newtheorem{theorem}[definition]{Theorem}
\newtheorem{lemma}[definition]{Lemma}

\newtheorem{remark}{Remark}[section]

\pagebreak

\title[Semilinear damped wave equation]{
Blow-up of solutions to the one-dimensional semilinear wave equation
with damping depending on time and space variables
}
\author{Yuta WAKASUGI}
\email{y-wakasugi@cr.math.sci.osaka-u.ac.jp}
\address{Department of Mathematics, Graduate School of Science,
Osaka University, Toyonaka, Osaka, 560-0043, Japan}
\begin{document}
\begin{abstract}
In this paper, we give a small data blow-up result for
the one-dimensional semilinear wave equation with damping
depending on time and space variables.
We show that if the damping term can be regarded as perturbation,
that is,
non-effective damping in a certain sense,
then the solution blows up in finite time for any power of nonlinearity.
This gives an affirmative answer for the conjecture that
the critical exponent agrees with that of the wave equation
when the damping is non-effective in one space dimension.
\end{abstract}

\maketitle
\section{Introduction}
We consider the initial value problem of
the one-dimensional semilinear damped wave equation
\begin{align}
\label{eq11}
	\left\{ \begin{array}{ll}
		u_{tt}-u_{xx}+a(t,x)u_t=|u|^p,&(t,x)\in (0,\infty)\times\mathbf{R},\\
		(u,u_t)(0,x)=(u_0,u_1)(x),&x\in \mathbf{R},
	\end{array}\right.
\end{align}
where
$u=u(t,x)$
is real-valued unknown and
$p>1$.
We assume that
$a(t,x)\in C^2([0,\infty)\times\mathbf{R})$
satisfies
\begin{equation}
\label{asa1}
	|\partial_t^{\alpha}\partial_x^{\beta}a(t,x)|
	\le \frac{\delta}{(1+t)^{k+\alpha}}\quad (\alpha,\beta=0,1)
\end{equation}
with some
$k>1$
and small
$\delta>0$.

In this paper we will prove a small data blow-up result,
that is,
if the initial data satisfy a certain condition,
which is independent of their amplitude,
then the corresponding solution blows up in finite time.

For the $n$-dimensional semilinear damped wave equation
\begin{equation}
	\label{DWn}
	u_{tt}-\Delta u+a(t,x)u_t=|u|^p,
\end{equation}
the {\em critical exponent} 
$p_c$
is well studied.
Here ``critical" means that if
$p_c < p$,
all small data solutions of \eqref{eq11} are global;
if
$1 < p \le p_c$,
the local solution cannot be extended globally even for small data.

When $a\equiv 1$,
Todorova and Yordanov \cite{TY1} and Zhang \cite{Z}
determined $p_c=1+2/n$.
Ikehata, Todorova and Yordanov \cite{ITY}
treated the case
$a\sim \langle x\rangle^{-\alpha}$
with
$0\le \alpha <1$
and proved that $p_c=1+2/(n-\alpha)$.
For the time-dependent case
$a=(1+t)^{-\beta}$,
Lin, Nishihara and Zhai \cite{LNZ}
determined
$p_c=1+2/n$
(see also \cite{DLR, DL} for more general types of damping
depending on the time variable).

However, there are only few results for the damping depending on
both time and space variables.
The author \cite{Wa1} considered
$a=\langle x\rangle^{-\alpha}(1+t)^{-\beta}$
with
$0\le \alpha, \beta$
and
$0\le \alpha+\beta<1$.
In this case it is conjectured that
$p_c$
is given by
$1+2/(n-\alpha)$.
He proved that if
$p>1+2/(n-\alpha)$,
then there is a unique global solution for any small data.
Recently, a similar result is obtained by Khader \cite{Kh} independently.
However, we do not know
whether the solution blows-up
in finite time when
$1<p\le 1+2/(n-\alpha)$.

In the previous results \cite{Z, ITY, LNZ},
the blow-up parts were obtained by a test-function method
developed by \cite{Z}.
As we will see in Section 3,
in order to apply the test function method,
we have to transform the equation \eqref{eq11} into divergence form
and the nonlinear term must be positive.
When the damping term depends only on the time variable,
Lin, Nishihara and Zhai \cite{LNZ}
used a positive solution $g(t)$ of an appropriate ordinary differential equation
and transformed the equation into divergence form.

Turning back to our problem,
we follow \cite{LNZ} and transform the equation (\ref{eq11}) into divergence form.
Multiplying \eqref{eq11} by a positive function
$g=g(t,x)$,
we have
\begin{equation}
\label{eqgu}
	(gu)_{tt}-(gu)_{xx}+2(g_xu)_x+((-2g_t+ga)u)_t
	+(g_{tt}-g_{xx}-(ga)_t)u=g|u|^p.
\end{equation}
Thus, if $g$ satisfies
\begin{equation}
\label{eqg}
	g_{tt}-g_{xx}-(ga)_t=0,
\end{equation}
then (\ref{eqgu}) becomes divergence form
and we can apply the test function method.
We will find a solution $g$ of (\ref{eqg}) having the form
\begin{equation}
\label{eqgh}
	g(t,x)=1+h(t,x),
\end{equation}
where $h$ has small amplitude, more precisely,
$|h(t,x)|\le \theta$
with some
$\theta\in (0,1)$.
This ensures the positivity of
$g$
and so the nonlinearity
$g|u|^p$.
Then
$h$
must satisfy
\begin{equation}
\label{eqh}
	h_{tt}-h_{xx}-a(t,x)h_t-a_t(t,x)(1+h)=0.
\end{equation}
We can find a classical solution
$h$
of \eqref{eqh}
having desired property by the method of characteristics.

\begin{lemma}\label{lem1}%%%
Let
$\theta\in (0,1)$
and
$k>1$.
Then there exists
$\delta_0>0$
such that for all $\delta\in (0,\delta_0)$
the following holds:
if $a$ satisfies \eqref{asa1}, then
there exists a solution $h\in C^2([0,\infty)\times \mathbf{R})$
of \eqref{eqh} satisfying
\begin{equation}
\label{esh}
	|h(t,x)|\le \frac{\theta}{(1+t)^{k-1}},\quad
	|\partial_{t}^{\alpha}\partial_x^{\beta}h(t,x)|\le \frac{C}{(1+t)^k}
	\quad (\alpha+\beta= 1)
\end{equation}
for all
$(t,x)\in [0,\infty)\times \mathbf{R}$
with some constant
$C>0$.
\end{lemma}

Using this $h$,
we can obtain a blow-up result for \eqref{eq11}.
To state our result precisely, we define the solution of \eqref{eq11}.
Let
$T\in (0,\infty]$.
We say that
$u\in X(T):=C([0,T);H^1(\mathbf{R}))\cap C^1([0,T);L^2(\mathbf{R}))$
is a solution of the initial value problem \eqref{eq11}
on the interval
$[0,T)$
if it holds that
\begin{align}
\label{sol}
	\lefteqn{\int_{[0,T)\times \mathbf{R}}
		u(t,x)(\partial_t^2\psi(t,x)-\partial_x^2\psi(t,x)-\partial_t(a(t,x)\psi(t,x)))dxdt}\notag\\
	&=\int_{\mathbf{R}}\left\{(a(0,x)u_0(x)+u_1(x))\psi(0,x)-u_0(x)\partial_t\psi(0,x)\right\}dx\\
	&\quad+\int_{[0,T)\times\mathbf{R}}|u(t,x)|^p\psi(t,x)dxdt\notag
\end{align}
for any
$\psi\in C_0^{2}([0,T)\times \mathbf{R})$.
In particular, when
$T=\infty$,
we call
$u$
a global solution.

We first recall a local existence result:
\begin{proposition}\label{prop1}%%%%%%
Let
$1<p<\infty$
and
$(u_0,u_1)\in H^1(\mathbf{R})\times L^2(\mathbf{R})$.
Then there exists $T^*\in (0,+\infty]$ and a unique solution
$u\in X(T^*)$.
Moreover, if
$T^*<+\infty$,
then it follows that
$$
	\lim_{t\to T^*-0}\|(u, u_t)(t)\|_{H^1\times L^2}=+\infty.
$$
\end{proposition}

For the proof, see for example \cite{IT}.
We put an assumption on the data
\begin{equation}
\label{eqini}
	\liminf_{R\to\infty}
	\int_{-R}^R((-g_t(0,x)+g(0,x)a(0,x))u_0(x)+g(0,x)u_1(x))dx>0,
\end{equation}
where $g$ is defined by (\ref{eqgh})
with
$h$
in Lemma \ref{lem1}.
Our main result is the following.

\begin{theorem}\label{th1}%%%
Let
$1<p<\infty$.
Under the same situation as Lemma \ref{lem1},
let
$(u_0,u_1)\in H^1(\mathbf{R})\times L^2(\mathbf{R})$
satisfy \eqref{eqini}.
Then the local solution $u$ of \eqref{eq11}
blows up in finite time, that is,
$\lim_{t\to T^*-0}\|(u,u_t)(0)\|_{H^1\times L^2}=+\infty$
holds for some
$T^*\in (0,+\infty)$.
\end{theorem}

\begin{remark}
(i) For Lemma \ref{lem1},
our method does not work in higher dimensional cases
$n\ge 2$
and we have no idea to find an appropriate solution
$g$
of \eqref{eqg}.
(ii) We expect that the assumption on the smallness of
$\delta$
is removable.
\end{remark}

Theorem \ref{th1} is closely related to so-called {\em diffusion phenomenon},
which means that
the solution of the damped wave equation
\begin{equation}
\label{DW}
	u_{tt}-\Delta u+a(t,x)u_t=f(u),
\end{equation}
behaves like a solution of
the corresponding heat equation
$-\Delta v+a(t,x)v_t=f(v)$.
Here
$f(u)$
denotes a nonlinear term.

For the linear and constant coefficient case,
that is
$f(u)=0$ and $a\equiv 1$,
the asymptotic behavior of the solution was initiated by
Matsumura \cite{M}.
He showed that
some decay rates of solution are same as that of corresponding heat equation
and applied these estimates to semilinear problems.
After that more specific asymptotics were given by
\cite{HM, MN, HO, N1, Na}.
They showed that the asymptotic profile of solution is
actually given by that of the corresponding heat equation.

For the linear and variable coefficient case,
Mochizuki \cite{Mo}
proved that if $a(t,x)$ has bounded derivatives and satisfies
$a(t,x)\lesssim (1+|x|)^{-k}$
with some
$k>1$,
then the solution of \eqref{DW} is asymptotically equivalent to
a solution of the free wave equation $w_{tt}-\Delta w=0$.
This means that
if the damping term decays sufficiently fast,
then the friction becomes non-effective
and the equation recovers its hyperbolic structure.
Wirth \cite{W1, W2} considered time-dependent dampings,
for example,
$a=(1+t)^{-k}$.
He showed that
if $k>1$ (resp. $k<1$), then the asymptotic profile of solution is given by
that of the free wave equation (resp. the corresponding heat equation).
When the space-dependent damping case
$a=a(x)\sim \langle x\rangle^{-\alpha}$
with $0\le\alpha <1$,
Todorova and Yordanov \cite{TY3}
obtained an energy decay estimate
of the solution by a weighted energy method.
The decay rate they obtained is almost same
as that of the corresponding heat equation.
This shows that in this case the equation has a diffusive structure
(see also \cite{ITY2, KK}).

For the semilinear case
$f(u)=|u|^p$,
the results \cite{TY1, Z, ITY, LNZ} we mentioned above
show the critical exponents coincide those of the corresponding
semilinear heat equations
(see \cite{F} for the heat equation).

On the contrary,
by the results of \cite{Mo, W1},
it is expected that
when the damping term decays sufficiently fast,
the critical exponent agrees with that of the wave equation.
However, we do not know any results for this conjecture
(see \cite{Wa2} for a partial result).
In particular,
in one-dimensional case,
Kato \cite{K} proved that
the critical exponent of the wave equation is given by $+\infty$,
that is,
the blow-up result holds for any
$1<p<\infty$.
Therefore, Theorem \ref{th1} can be interpreted as an affirmative answer
for this conjecture in one-dimensional case.

We can also treat other types of damping.
We give two examples.
These examples have the shape
$a(t,x)=\mu/(1+t)+b(t,x)$,
here
$b$
denotes a perturbation term.
The wave equation with the damping term
$\frac{\mu}{1+t}u_t$
was investigated by
\cite{Wi04, D, DL, Wa2}.
Wirth \cite{Wi04} obtained several
$L^p$-$L^q$
estimates of solutions to the linear problem.
Using these estimates,
recently,
D'Abbicco \cite{D} proved several global existence results
(see also \cite{DL} for blow-up results).
The author \cite{Wa2} also obtained a certain global existence result
by a weighted energy method.

The first example is the case that
$a(t,x)$
is a perturbation of
$2/(1+t)$,
that is,
$a$
is given by
\begin{equation}
\label{eqa2}
	a(t,x)=\frac{2}{1+t}+b(t,x)
\end{equation}
and $b(t,x)\in C^2([0,\infty)\times\mathbf{R})$ satisfies \eqref{asa1}.

In this case by putting
\begin{equation}
\label{eqg2}
	g(t,x)=(1+t)(1+h(t,x)),
\end{equation}
the equation \eqref{eqh} becomes
\begin{equation}
\label{eqh2}
	h_{tt}-h_{xx}-b(t,x)h_t-\left(\frac{b(t,x)}{1+t}+b_t(t,x)\right)(1+h)=0.
\end{equation}
In the same way as in the proof of Lemma 1.1,
we can obtain a solution
$h$
of \eqref{eqh2}:
\begin{lemma}\label{lem2}%%%
Let
$\theta\in (0,1)$
and
$k>1$.
Then there exists
$\delta_0>0$
such that for all $\delta\in (0,\delta_0)$
the following holds:
if $b$ satisfies \eqref{asa1} and
$a$
is given by \eqref{eqa2}, then
there exists a solution $h\in C^2([0,\infty)\times \mathbf{R})$
of (\ref{eqh2}) satisfying
\begin{equation}
\label{esh2}
	|h(t,x)|\le \frac{\theta}{(1+t)^{k-1}},\quad
	|\partial_{t}^{\alpha}\partial_x^{\beta}h(t,x)|\le \frac{C}{(1+t)^k}
	\quad (\alpha+\beta= 1)
\end{equation}
for all
$(t,x)\in [0,\infty)\times\mathbf{R}$
with some constant
$C>0$.
\end{lemma}

Using this $h$,
we can apply a test function method and obtain
a blow-up result.
\begin{theorem}\label{th2}%%%%%%%
Let
$1<p\le 3$.
Under the same situation as Lemma \ref{lem2},
let $(u_0,u_1)\in H^1(\mathbf{R})\times L^2(\mathbf{R})$
satisfy \eqref{eqini}.
Then the local solution $u$ of \eqref{eq11}
blows up in finite time.
\end{theorem}

The second example is
\begin{equation}
\label{eqa3}
	a(t,x)=\frac{\mu}{1+t}+b(t,x)
\end{equation}
with
$\mu>0$
and
$b(t,x)\in C^2([0,\infty)\times \mathbf{R})$
satisfying \eqref{asa1}.
By putting
\begin{equation}
\label{eqg3}
	g(t,x)=(1+t)^{\mu}(1+h(t,x)),
\end{equation}
we have
\begin{equation}
\label{eqh3}
	h_{tt}-h_{xx}+\left(\frac{\mu}{1+t}-b\right)h_t
	-\left(\frac{\mu}{1+t}b+b_t\right)(1+h)=0.
\end{equation}
In a similar way to Lemma \ref{lem1}
with some technical argument,
we can find an appropriate solution
$h$
of \eqref{eqh3}.

\begin{lemma}\label{lem3}%%%%%
Let
$\theta\in (0,1), \mu>0$
and
$k>\max\{ 1, \mu \}$.
Then there exists
$\delta_0>0$
such that for all $\delta\in (0,\delta_0)$
the following holds:
if $b$ satisfies \eqref{asa1} and
$a$
is given by \eqref{eqa3}, then
there exists a solution $h\in C^2([0,\infty)\times \mathbf{R})$
of (\ref{eqh3}) satisfying
\begin{equation}
	|h(t,x)|\le \frac{\theta}{(1+t)^{k-1}},\quad
	|\partial_{t}^{\alpha}\partial_x^{\beta}h(t,x)|\le \frac{C}{(1+t)^k}
	\quad (\alpha+\beta= 1)
\end{equation}
for all
$(t,x)\in [0,\infty)\times\mathbf{R}$
with some constant
$C>0$.
\end{lemma}
This lemma and the test function method imply

\begin{theorem}\label{th3}%%%%%
Let
$1<p\le 1+2/\mu$
Under the same situation as Lemma \ref{lem3},
let $(u_0,u_1)\in H^1(\mathbf{R})\times L^2(\mathbf{R})$
satisfy \eqref{eqini}.
Then the local solution $u$ of \eqref{eq11}
blows up in finite time.
\end{theorem}

\begin{remark}
When
$\mu=2$,
Theorem \ref{th2}
is better than Theorem \ref{th3}.
\end{remark}

At the end of this section, we explain some notation and terminology
used throughout this paper.
We put
$$
	\|f\|_{L^p(\mathbf{R}^n)}:=
	\left(\int_{\mathbf{R}^n}|f(x)|^pdx\right)^{1/p}
$$
for
$1<p<\infty$
and
$\|f\|_{\infty}:=\sup_{x\in\mathbf{R}}|f(x)|$.
We denote the usual Sobolev space by
$H^1(\mathbf{R})$.
For an interval
$I$
and a Banach space
$X$,
we define
$C^r(I;X)$
as the Banach space whose
element is an
$r$-times continuously differentiable mapping from
$I$
to
$X$
with respect to the topology in
$X$
(if $I$ is semi-open or closed interval,
the differential at the endpoint is interpreted as one-sided derivative).
The letter $C$ indicates the generic constant, which may change from line to line. 
We also use the symbols $\lesssim$ and $\sim$.
The relation $f\lesssim g$ means $f\le Cg$ with some constant $C>0$
and $f\sim g$ means $f\lesssim g$ and $g\lesssim f$.

%%%%%%%%%%%%%%
\section{Proof of Lemma \ref{lem1}}
In this section, we construct a solution of (\ref{eqh})
by the method of characteristics.
First, we diagonalize (\ref{eqh}).
Put
$$
	\left(\begin{array}{c}v_1\\ v_2\end{array}\right)
	=\left(\begin{array}{c}h_t+h_x\\ h_t-h_x\end{array}\right).
$$
Then $v_1, v_2$ satisfies
\begin{equation}
\label{eq21}
	\partial_tv_1=
	\partial_xv_1+\frac{a(t,x)}{2}(v_1+v_2)+a_t(t,x)(1+h)
\end{equation}
and
\begin{equation}
\label{eq22}
	\partial_tv_2=
	-\partial_xv_2+\frac{a(t,x)}{2}(v_1+v_2)+a_t(t,x)(1+h),
\end{equation}
respectively.
We can rewrite \eqref{eq21}-\eqref{eq22} as
\begin{align*}
	\partial_t(v_1(t,x-t))&=\frac{a(t,x-t)}{2}(v_1(t,x-t)+v_2(t,x-t))+a_t(t,x-t)(1+h(t,x-t)),\\
	\partial_t(v_2(t,x+t))&=\frac{a(t,x+t)}{2}(v_1(t,x+t)+v_2(t,x+t))+a_t(t,x+t)(1+h(t,x+t)).
\end{align*}
We seek solutions satisfying
$\lim_{t\to +\infty}(v_1, v_2)=0$
uniformly in
$x$.
Integrating the above identities over
$[t,\infty)$
and changing variables,
one has a system of integral equation
\begin{align}
\label{eq23}
	v_1(t,x)&=-\int_t^{\infty}\left\{
		\frac{a}{2}(v_1+v_2)+a_t(1+h)\right\}(s,x+t-s)ds,\\
\label{eq24}v_2(t,x)&=-\int_t^{\infty}\left\{
	\frac{a}{2}(v_1+v_2)+a_t(1+h)\right\}(s,x-(t-s))ds.
\end{align}
Next, we construct solutions to \eqref{eq23}, \eqref{eq24} by
an iteration argument in an appropriate Banach space.
We define a function space
$Y$.
We say
$V=(v_1,v_2,h)\in Y$
if
$V\in \left(C([0,\infty)\times\mathbf{R})\right)^3$,
$V$
is differentiable with respect to
$x$
for all
$(t,x)\in [0,\infty)\times\mathbf{R}$,
$\partial_xV\in \left(C([0,\infty)\times\mathbf{R})\right)^3$,
and
$\|V\|_Y=\|(v_1,v_2,h)\|_Y<+\infty$,
where
\begin{align*}
	\|(v_1,v_2,h)\|_Y&=\sup_{t\in [0,\infty)}\left\{
		(1+t)^k\|v_1(t)\|_{\mathcal{B}^1}+(1+t)^k\|v_2(t)\|_{\mathcal{B}^1}
		+(1+t)^{k-1}\|h(t)\|_{\mathcal{B}^1}\right\},\\
	\|h(t)\|_{\mathcal{B}^1}&=\|h(t,\cdot)\|_{\infty}
		+\|\partial_xh(t,\cdot)\|_{\infty}.
\end{align*}
Then
$Y$
is a Banach space with norm
$\|V\|_Y$.
Let
$\theta\in (0,1)$
and let
\begin{align*}
	K_{\theta}=&\{ (v_1,v_2,h)\in Y \mid
	\sup_{t\in[0,\infty)}(1+t)^k\|v_1(t)\|_{\infty}\le\theta,\\
	&\sup_{t\in [0,\infty)}(1+t)^k\|v_2(t)\|_{\infty}\le\theta,\ \ 
	\sup_{t\in [0,\infty)}(1+t)^k\|h(t)\|_{\infty}\le\theta
	\}.
\end{align*}
Take
$(v_1^{(0)}, v_2^{(0)}, h^{(0)} )\in K_{\theta}$
arbitrarily and define
$V^{(n)}=(v_1^{(n)}, v_2^{(n)}, h^{(n)})$
inductively by
\begin{align}
\label{eq25}
	v_1^{(n)}(t,x)&=
		-\int_t^{\infty}
		\left\{\frac{a}{2}(v_1^{(n-1)}+v_2^{(n-1)})+a_t(1+h^{(n-1)})\right\}(s,x+t-s))ds,\\
\nonumber
	v_2^{(n)}(t,x)&=
		-\int_t^{\infty}
		\left\{\frac{a}{2}(v_1^{(n-1)}+v_2^{(n-1)})+a_t(1+h^{(n-1)})\right\}(s,x-(t-s))ds,\\
\nonumber
	h^{(n)}(t,x)&=-\frac{1}{2}\int_t^{\infty}(v_1^{(n)}+v_2^{(n)})(s,x)ds.
\end{align}
The following proposition
shows that
if the coefficient of damping term $\delta$
is sufficiently small,
then
$\{ V^{(n)} \}_{n=0}^{\infty}$
is a Cauchy sequence.

\begin{proposition}\label{Prop21}%%%
Let
$k>1$
and
$\theta\in (0,1)$.
Then there exists
$\delta_0>0$
such that for any
$\delta\in (0, \delta_0]$
the following holds:
if
$a$
satisfies \eqref{asa1},
then
$\{ V^{(n)} \}_{n=0}^{\infty}\in K_{\theta}$
for all
$n$
and
$\{ V^{(n)} \}_{n=0}^{\infty}$
is a Cauchy sequence with respect to the norm $\| \cdot \|_Y$.
\end{proposition}
\begin{proof}%
Frist, we prove that
if
$V^{(n-1)}\in K_{\theta}$,
then
$V^{(n)}\in K_{\theta}$.
Assume
$V^{(n-1)}\in K_{\theta}$.
It is obvious that
$V^{(n)}\in \left(C([0,\infty)\times\mathbf{R})\right)^3$.
We have
\begin{align*}
	|v_1^{(n)}(t,x)|&\le
		\int_t^{\infty}
		\left\{\frac{|a|}{2}(|v_1^{(n-1)}+|v_2^{(n-1)}|)
		+|a_t|(1+|h|)\right\}(s,x+t-s)ds\\
	&\le \int_t^{\infty}\frac{\theta\delta}{(1+s)^{2k}}
		+\frac{\delta}{(1+s)^{k+1}}(1+\theta (1+s)^{-(k-1)})ds\\
	&=\delta C_1(1+t)^{-k}
\end{align*}
with some
$C_1>0$.
Hereafter,
$C_j\ (j=1,2,\ldots)$
denotes a constant depending only on
$k, \theta$.
Moreover, differentiating under the integral sign,
we obtain
\begin{align*}
	\partial_xv_1^{(n)}(t,x)=
		-\int_t^{\infty}
		&\left\{ \frac{a_x}{2}(v_1^{(n-1)}+v_2^{(n-1)})
		+\frac{a}{2}(\partial_xv_1^{(n-1)}+\partial_xv_2^{(n-1)})\right.\\
	&\left.+a_{tx}(1+h^{(n-1)})+a_th_x^{(n-1)}\right\}(s,x+t-s)ds.
\end{align*}
This implies
$\partial_xv_1^{(n)}\in C([0,\infty)\times\mathbf{R})$
and
\begin{align*}
	|\partial_xv_1^{(n)}(t,x)|
		&\le \delta C_2(1+t)^{-k}.
\end{align*}
We can also obtain the same estimates for $v_2^{(n)}$.
By differentiating under the integral sign again, we have
$$
	\partial_xh^{(n)}(t,x)
	=-\frac{1}{2}\int_t^{\infty}
	(\partial_xv_1^{(n)}+\partial_xv_2^{(n)})(s,x)ds.
$$
Thus, we have
\begin{align*}
	|h^{(n)}(t,x)|&\le \delta C_1\int_t^{\infty}\frac{ds}{(1+s)^k}=\delta C_3(1+t)^{-(k-1)},\\
	|\partial_xh^{(n)}(t,x)|&\le \delta C_2\int_t^{\infty}\frac{ds}{(1+t)^k}=\delta C_4(1+t)^{-(k-1)}.
\end{align*}
The above estimates show
$V^{(n)}\in Y$.
Moreover,
taking
$\delta_0$
so small that
$\delta_0\max\{ C_1, C_3 \}\le \theta$,
we have
$V^{(n)}\in K_{\theta}$
for all
$\delta\in (0,\delta_0]$.

Next, we prove that
$\{ V^{(n)} \}_{n=0}^{\infty}$
is a Cauchy sequence with respect to the norm
$\|\cdot\|_Y$.
It follows that
\begin{align*}
	|v_1^{(n)}(t,x)-v_1^{(n-1)}(t,x)|
	&\le \int_t^{\infty}
		\left\{\frac{\delta}{2(1+s)^k}(|v_1^{(n-1)}-v_1^{(n-2)}|+|v_2^{(n-1)}-v_2^{(n-2)}|)\right.\\
		&\quad\left.+\frac{\delta}{(1+s)^{k+1}}|h^{(n-1)}-h^{(n-2)}|\right\}(s,x+t-s)ds\\
	&\le \delta C_5(1+t)^{-(2k-1)}\|V^{(n-1)}-V^{(n-2)}\|_Y.
\end{align*}
In the same way, we have
\begin{align*}
	|\partial_xv_1^{(n)}(t,x)-\partial_xv_1^{(n-1)}(t,x)|
	&\le \delta C_6(1+t)^{-(2k-1)}\|V^{(n-1)}-V^{(n-2)}\|_Y.
\end{align*}
and the same estimates holds for $v_2^{(n)}-v_2^{(n-1)}$.
We also obtain
\begin{align*}
	|h^{(n)}(t,x)-h^{(n-1)}(t,x)|
		&\le \delta C_7(1+t)^{-2(k-1)}\|V^{(n-1)}-V^{(n-2)}\|_Y,\\
	|\partial_xh^{(n)}(t,x)-\partial_xh^{(n-1)}(t,x)|
		&\le \delta C_8(1+t)^{-2(k-1)}\|V^{(n-1)}-V^{(n-2)}\|_Y.
\end{align*}
Consequently, taking
$\delta_0$
smaller so that
$r=\delta_0 (2C_5+2C_6+C_7+C_8)<1$,
we have
$$
	\|V^{(n)}-V^{(n-1)}\|_Y\le r\|V^{(n-1)}-V^{(n-2)}\|_Y,
$$
which shows that
$\{ V^{(n)} \}_{n=0}^{\infty}$
is a Cauchy sequence.
\end{proof}

\begin{proof}[Proof of Lemma \ref{lem1}]

By the above proposition,
$\{ V^{(n)} \}_{n=0}^{\infty}$
is a Cauchy sequence and converges to some element
$(v_1, v_2, h)\in K_{\theta}$.
Therefore,
$(v_1, v_2, h)$
satisfies the integral equation \eqref{eq23}, \eqref{eq24}.
By noting the differentiability with respect to
$t$
of the right-hand-side of \eqref{eq23} \eqref{eq24},
we have
$v_1, v_2\in C^1([0,\infty)\times\mathbf{R})$.
Differentiating \eqref{eq23} and \eqref{eq24}, one can see that
$v_1, v_2$ satisfies the differential equation \eqref{eq21}, \eqref{eq22}, respectively.
By the equation of
$h$,
we also have
\begin{equation}
\label{eq26}
	\partial_th(t,x)=\frac{1}{2}(v_1(t,x)+v_2(t,x)),\quad
	\partial_xh(t,x)=\frac{1}{2}(v_1(t,x)-v_2(t,x)).
\end{equation}
Thus, $h\in C^2([0,\infty)\times \mathbf{R})$
and $h$ is a classical solution of \eqref{eqh}.
By $(v_1, v_2, h)\in K_{\theta}$ and \eqref{eq26}, the estimate \eqref{esh} is obvious.
\end{proof}

%%%%%%
\section{Proof of Theorem \ref{th1}}
In this section, we give a proof of Theorem \ref{th1}.
By Lemma \ref{lem1}, there exists
$h$
satisfying \eqref{eqh}.
Thus, \eqref{eqg} holds for
$g$
given by \eqref{eqgh}
and we can transform the equation \eqref{eq11} into divergence form
\begin{equation}
\label{eq31}
	(gu)_{tt}-(gu)_{xx}+2(g_xu)_x+((-2g_t+ga)u)_t=g|u|^p.
\end{equation}
We apply a text function method to \eqref{eq31}.
Since
$g$
is defined by \eqref{eqgh}
and
$h$ 
satisfies \eqref{esh}, we have
\begin{equation}
\label{eq32}
	C^{-1}\le g(t,x)\le C,\quad
	|g_t(t,x)|\le \frac{C}{(1+t)^k},\quad
	|g_x(t,x)|\le \frac{C}{(1+t)^k}
\end{equation}
with some constant $C>0$.
We define test functions
\begin{align*}
	\phi(x)&=\left\{\begin{array}{ll}
		1&(|x|\le 1/2\\
		\displaystyle\frac{\exp(-1/(1-x^2))}{\exp(-1/(x^2-1/4))+\exp(-1/(1-x^2))}&(1/2<|x|<1),\\
		0&(|x|\ge 1),
	\end{array}\right.\\
	\eta(t)&=\left\{\begin{array}{ll}
		1&(0\le t\le 1/2),\\
		\displaystyle\frac{\exp(-1/(1-t^2))}{\exp(-1/(t^2-1/4))+\exp(-1/(1-t^2))}&(1/2<t<1),\\
		0&(t\ge 1).
	\end{array}\right.
\end{align*}
It is obvious that
$\phi\in C_0^{\infty}(\mathbf{R}), \eta\in C_0^{\infty}([0,\infty))$.
We also see that
\begin{align}
\label{eq33}
	|\phi^{\prime}(x)|&\lesssim \phi(x)^{1/p},\quad
	|\phi^{\prime\prime}(x)|\lesssim \phi(x)^{1/p},\\
	|\eta^{\prime}(t)|&\lesssim \eta(t)^{1/p},\quad
	|\eta^{\prime\prime}(t)|\lesssim \eta(t)^{1/p}.\notag
\end{align}
Indeed, let
$q,r$
satisfy
$1/p+1/q=1, 1/p+2/r=1$
and let
$\mu =\phi^{1/q}, \nu=\phi^{1/r}$.
Then we have
$$
	|\phi^{\prime}|=|(\mu^q)^{\prime}|
	=|q\mu^{q-1}\mu^{\prime}|\lesssim \mu^{q-1}=\phi^{1/p}
$$
and
$$
	|\phi^{\prime\prime}|=|(\nu^r)^{\prime}|
	\lesssim |\nu^{\prime\prime}|\nu^{r-1}+|\nu^{\prime}|^2\nu^{r-2}
	\lesssim \nu^{r-2}=\phi^{1/p}.
$$
To prove Theorem \ref{th1},
we use a contradiction argument.
Suppose
$u\in X(\infty)$
is a global solution to \eqref{eq11}
with initial data
$(u_0, u_1)$
satisfying \eqref{eqini}.
Let $\tau, R$ be parameters such that
$\tau\in (\tau_0, \infty), R\in (R_0, \infty)$,
where
$\tau_0\ge 1, R_0>0$
are defied later.
We put
\begin{align*}
	\eta_{\tau}(t)=\eta(t/\tau),\quad \phi_R(x)=\phi(x/R),\\
	\psi_{\tau, R}(t,x)=\eta_{\tau}(t)\phi_R(x)
\end{align*}
and
\begin{align*}
	I_{\tau,R}&:=\int_0^{\tau}\int_{-R}^Rg|u|^p\psi_{\tau,R}dxdt,\\
	J_R&:=\int_{-R}^R
		\left( (-g_t(0,x)+g(0,x)a(0,x))u_0(x)+g(0,x)u_1(x)\right)\phi_R(x)dx,
\end{align*}
Substituting the test function
$g(t,x)\psi_{\tau, R}(t,x)$
into the definition of solution \eqref{sol},
we see that
\begin{align*}
	I_{\tau,R}+J_R&=
	\int_0^{\tau}\int_{-R}^R\left( gu\partial_t^2\psi_{\tau,R}-gu\partial_x^2\psi_{\tau,R}
		-2(g_xu)\partial_x\psi_{\tau,R}\right.\\
	&\left.-(-2g_t+ga)u\partial_t\psi_{\tau,R}\right)dxdt
	=:K_1+K_2+K_3+K_4.
\end{align*}

Next, we estimate the terms $K_1,\ldots, K_4$.
Let
$q$
be the dual of
$p$,
that is
$q=p/(p-1)$.
By using the H\"older inequality and \eqref{eq32}, \eqref{eq33}, it follows that
\begin{align*}
	K_1
	&\le \tau^{-2}\int_0^{\tau}\int_{-R}^R|gu||\eta^{\prime\prime}(t/\tau)|\phi_R(x)dxdt\\
	&\lesssim \tau^{-2}I_{\tau,R}^{1/p}
		\left(\int_{-R}^R\left(\int_{0}^{\tau}g(t,x)dt\right)\phi_R(x)dx\right)^{1/q}\\
	&\lesssim \tau^{-2+1/q}R^{1/q}I_{\tau,R}^{1/p},
\end{align*}
\begin{align*}
	K_2&\le
	R^{-2}\int_0^{\tau}\int_{-R}^R|gu||\phi^{\prime\prime}(x/R)|\eta_{\tau}(t)dxdt\\
	&\lesssim R^{-2}I_{\tau,R}^{1/p}
		\left(\int_{-R}^R\left(\int_0^{\tau}g(t,x)\eta_{\tau}(t)dt\right)dx\right)^{1/q}\\
	&\lesssim \tau^{1/q}R^{-2+1/q}I_{\tau,R}^{1/p},
\end{align*}
\begin{align*}
	K_3&\le R^{-1}\int_0^{\tau}\int_{-R}^R|g_xu||\phi^{\prime}(x/R)|\eta_{\tau}dxdt\\
	&\lesssim R^{-1}I_{\tau,R}^{1/p}
		\left(\int_{-R}^R\left(\int_0^{\tau}(1+t)^{-qk}dt\right)dx\right)^{1/q}\\
	&\lesssim R^{-1+1/q}I_{\tau,R}^{1/p}.
\end{align*}
Finally, we estimate $K_4$.
Noting that
${\rm supp}\, \eta^{\prime}(t)\subset [1/2, 1]$,
we have
\begin{align*}
	K_4&\le
		\tau^{-1}\int_0^{\tau}\int_{-R}^R(2|g_t|+|ga|)|u||\eta^{\prime}(t/\tau)|\phi_Rdxdt\\
	&\lesssim \tau^{-1}I_{\tau,R}^{1/p}
		\left(\int_{-R}^R\left(\int_{\tau/2}^{\tau}(1+t)^{-kq}dt\right)\phi_R(x)dx\right)^{1/q}\\
	&\lesssim \tau^{-1-k+1/q}R^{1/q}I_{\tau,R}^{1/p}.
\end{align*}

Therefore, putting
$$
	D(\tau,R):=\tau^{-2+1/q}R^{1/q}+\tau^{1/q}R^{-2+1/q}+R^{-1+1/q},
$$
we obtain
\begin{equation}
\label{eq34}
	I_{\tau, R}+J_R\le CD(\tau,R)I_{\tau,R}^{1/p}.
\end{equation}
By the assumption on the data \eqref{eqini},
there exists $R_0>0$ such that
$J_R>0$
holds for
$R\ge R_0$.
This implies
$$
	I_{\tau,R}\le CD(\tau,R)^q.
$$
Putting
$\tau_0=R_0$
and
$R=\tau$,
we have
\begin{equation}
\label{eqIt2}
	I_{\tau,\tau}\le C\tau^{-1+1/q}
\end{equation}
for $\tau\ge \tau_0$.
In particular,
$I_{\tau,\tau}\le C$
with some $C>0$
and hence,
$g|u|^p\in L^1([0,\infty)\times\mathbf{R})$
and
$\lim_{\tau\to +\infty}I_{\tau,\tau}=\|g|u|^p\|_{L^1([0,\infty)\times\mathbf{R})}$.
Moreover, since
$-1+1/q<0$,
by letting $\tau\rightarrow +\infty$,
the right-hand-side of \eqref{eqIt2} tends to
$0$.
This gives 
$\|g|u|^p\|_{L^1([0,\infty)\times\mathbf{R})}=0$,
that is
$u\equiv 0$.
However,
in view of \eqref{sol},
it contradicts $(u_0, u_1)\neq 0$.
This completes the proof.

\section{Proof of Theorem \ref{th2}}%%%%
In this section, we give a proof of Theorem \ref{th2}.
Note that we can prove Lemma \ref{lem2}
by the same argument as the proof of Lemma \ref{lem1}.
The only difference between the proofs of Lemmas \ref{lem1} and \ref{lem2}
is that of the coefficients of
\eqref{eqh} and \eqref{eqh2}.
However, by the assumption on $b(t,x)$,
it follows that
$$
	|b(t,x)|\le \frac{\delta}{(1+t)^k},\quad
	\left| \frac{b(t,x)}{1+t}+b_t(t,x)\right| \le \frac{2\delta}{(1+t)^{k+1}}.
$$
Using this estimate, we can prove Lemma \ref{lem2}
in the same way as Section 2
and hence, we omit the detailed proof.

Now we prove Theorem \ref{th2}.
By \eqref{esh2} and \eqref{eqg2}, we have
\begin{equation}
\label{esg2}
	g\sim (1+t),\quad |g_t|\lesssim 1,\quad
	|g_x|\lesssim (1+t)^{-k+1}.
\end{equation}
We use the same notation as in Section 3
and suppose that $u$ is a global solution.
The main difference with the previous section
lies in the estimate of the terms $K_1, \ldots, K_4$.
In this case, we have
\begin{align*}
	K_1
	&\lesssim \tau^{-2+2/q}R^{1/q}I_{\tau,R}^{1/p},\\
	K_2
	&\lesssim \tau^{2/q}R^{-2+1/q}I_{\tau,R}^{1/p},\\
	K_3
	&\lesssim F(\tau)R^{-1+1/q}I_{\tau,R}^{1/p},\\
	K_4
	&\lesssim \tau^{-1-1/p+1/q}R^{1/q}I_{\tau,R}^{1/p}.
\end{align*}
where
\begin{equation}
\label{eqF}
	F(\tau)=\left\{\begin{array}{ll}
		1&(-q(1/p+k-1)<-1),\\
		(\log\tau)^{1/q}&(-q(1/p+k-1)=-1),\\
		\tau^{-(1/p+k-1)+1/q}&(-q(1/p+k-1)>-1)
	\end{array}\right.
\end{equation}
and we have used that
${\rm supp}\, \eta^{\prime}(t)\subset [1/2, 1]$
and \eqref{esg2} for the estimate of
$K_4$.
Let
\begin{equation}
\label{eqD2}
	D(\tau,R):=\tau^{-2+2/q}R^{1/q}+\tau^{2/q}R^{-2+1/q}
	+F(\tau)R^{-1+1/q}.
\end{equation}
We note that the powers of each terms of
$D(\tau, \tau)$
are negative if
$1<p< 3$.
Thus, by putting $R=\tau$
and the same argument as the previous section,
we can lead a contradiction.

When $p=3$,
we need a certain modification of the above argument.
We put
$$
	I_{\tau,R}^{\prime}=\int_{\tau/2}^{\tau}\int_{-R}^Rg|u|^p\psi_{\tau,R}dxdt,\quad
	I_{\tau,R}^{\prime\prime}=\int_0^{\tau}\int_{R/2<|x|<R}g|u|^p\psi_{\tau,R}dxdt.
$$
Then we can improve the estimates of
$K_1,\ldots,K_4$
as
\begin{align*}
	K_1&\le \tau^{-2+2/q}R^{1/q}(I_{\tau,R}^{\prime})^{1/p},\\
	K_2&\le \tau^{2/q}R^{-2+1/q}(I_{\tau,R}^{\prime\prime})^{1/p},\\
	K_3&\le F(\tau)R^{-1+1/q}(I_{\tau,R}^{\prime\prime})^{1/p},\\
	K_4&\le \tau^{-1-1/p+1/q}R^{1/q}(I_{\tau,R}^{\prime})^{1/p}.
\end{align*}
Thus, we have
\begin{align*}
	I_{\tau,R}\le C\left(\tau^{-2+2/q}R^{1/q}(I_{\tau,R}^{\prime})^{1/p}
		+(\tau^{2/q}R^{-2+1/q}+F(\tau)R^{-1+1/q})(I_{\tau,R}^{\prime\prime})^{1/p}
	\right).
\end{align*}
Substituting
$p=3$
and
$R=\tau$,
we obtain
\begin{equation}
\label{esI}
	I_{\tau,\tau}\le C((I_{\tau,\tau}^{\prime})^{1/3}+(I_{\tau,\tau}^{\prime\prime})^{1/3}).
\end{equation}
In particular, we see that
$I_{\tau, \tau}\le C$
with some constant
$C>0$,
since
$I_{\tau,\tau}^{\prime}\le I_{\tau,\tau}$
and
$I_{\tau,\tau}^{\prime\prime}\le I_{\tau,\tau}$.
Hence
$g|u|^3\in L^1([0,\infty)\times\mathbf{R})$
and
$\lim_{\tau\to\infty}I_{\tau,\tau}=\|g|u|^3\|_{L^1([0,\infty)\times\mathbf{R})}$.
However,
by noting the integral region of
$I_{\tau,\tau}^{\prime}, I_{\tau,\tau}^{\prime\prime}$,
we can see that the integrability of $g|u|^3$ shows
$$
	\lim_{\tau\to\infty}I_{\tau,\tau}^{\prime}=0,\quad
	\lim_{\tau\to\infty}I_{\tau,\tau}^{\prime\prime}=0.
$$
Therefore, turning back to \eqref{esI},
we obtain
$\lim_{\tau\to\infty}I_{\tau,\tau}=0$.
This implies
$u\equiv 0$.
In view of \eqref{sol},
this contradicts
$(u_0, u_1)\neq 0$.
This completes the proof.

\section{Proof of Theorem \ref{th3}}%%
In this section, we give a proof of Theorem \ref{th3}.
In order to prove Lemma \ref{lem3},
we modify the argument in Section 2.
Following the argument in Section 2,
we look for an appropriate solution by the following iteration:
\begin{align}
\label{eq51}
	v_1^{(n)}(t,x)&=
		\int_t^{\infty}\left\{\frac{1}{2}\left(\frac{\mu}{1+s}-b\right)
		(v_1^{(n-1)}+v_2^{(n-1)})\right.\\
\nonumber
		&\quad\left.-\left(\frac{\mu}{1+s}b+b_t\right)(1+h^{(n-1)})\right\}(s,x+t-s))ds,\\
\nonumber
	v_2^{(n)}(t,x)&=
		\int_t^{\infty}\left\{\frac{1}{2}\left(\frac{\mu}{1+s}-b\right)
		(v_1^{(n-1)}+v_2^{(n-1)})\right.\\
\nonumber
		&\quad\left.-\left(\frac{\mu}{1+s}b+b_t\right)(1+h^{(n-1)})\right\}(s,x-(t-s))ds,\\
\nonumber
	h^{(n)}(t,x)&=-\frac{1}{2}\int_t^{\infty}(v_1^{(n)}+v_2^{(n)})(s,x)ds.
\end{align}
We modify the definition of function space $Y$ in Section 2 as follows.
We say
$V=(v_1,v_2,h)\in Y$
if
$V\in \left(C([0,\infty)\times\mathbf{R})\right)^3$,
$V$
is differentiable with respect to
$x$
for all
$(t,x)\in [0,\infty)\times\mathbf{R}$,
$\partial_xV\in \left(C([0,\infty)\times\mathbf{R})\right)^3$,
and
$\|V\|_Y=\|(v_1,v_2,h)\|_Y<+\infty$,
where
\begin{align*}
	\|(v_1,v_2,h)\|_Y&=\sup_{t\in [0,\infty)}\left\{
		\lambda(1+t)^k\|v_1(t)\|_{\mathcal{B}^1}
		+\lambda(1+t)^k\|v_2(t)\|_{\mathcal{B}^1}
		+(1+t)^{k-1}\|h(t)\|_{\mathcal{B}^1}\right\},
\end{align*}
where
$\lambda$
is a large parameter fixed later.
We put
\begin{align}
\label{Ktheta}
	K_{\theta}:=&\{ (v_1,v_2,h)\in Y \mid
	\sup_{t\in [0,\infty)}(1+t)^k\|v_1(t)\|_{\infty}\le\theta^{\prime},\\
\nonumber
	&\sup_{t\in [0,\infty)}(1+t)^k\|v_2(t)\|_{\infty}\le\theta^{\prime},\quad
	\sup_{t\in [0,\infty)}(1+t)^{k-1}\|h(t)\|_{\infty}\le\theta
	\},
\end{align}
where
$\theta^{\prime}:=\theta\min\{ k-1, 1 \}$.
We take
$(v_1^{(0)}, v_2^{(0)}, h^{(0)} )\in K_{\theta}$
arbitrarily and define
$V^{(n)}=(v_1^{(n)}, v_2^{(n)}, h^{(n)} )$
inductively by \eqref{eq51}.
Now we prove that
$\{ V^{(n)} \}_{n=0}^{\infty}$
is a Cauchy sequence in
$K_{\theta}$
for sufficiently large
$\lambda$
and small
$\delta$.
\begin{proposition}\label{prop51}%%%%
If
$k>\max\{ 1, \mu \}$,
then there exist $\lambda$ and $\delta_0$
having the following property:
if $\delta\in (0,\delta_0]$,
then
$\{ V^{(n)} \}_{n=0}^{\infty}$
is a Cauchy sequence in
$K_{\theta}$
with respect to the norm
$\|\cdot\|_Y$.
\end{proposition}
\begin{proof}
We first show that if
$V^{(n-1)}\in K_{\theta}$,
then
$V^{(n)}\in K_{\theta}$.
We calculate
\begin{align*}
	|v_1^{(n)}(t,x)|\le \left(\frac{\mu}{k}\theta^{\prime}
		+\delta C\right)(1+t)^{-k}.
\end{align*}
In view of
$k>\mu$,
by taking
$\delta$
sufficiently small,
we obtain
$$
	(1+t)^k|v_1^{(n)}(t,x)|\le \theta^{\prime}.
$$
By the same way, we also have
$
	(1+t)^k|v_2^{(n)}(t,x)|\le \theta^{\prime}.
$
Noting that
$\theta^{\prime}/(k-1)\le \theta$,
we obtain
$$
	|h^{(n)}(t,x)|\le \int_t^{\infty}\frac{\theta^{\prime}}{(1+s)^k}ds
	\le \theta(1+t)^{-(k-1)}.
$$
By differentiating under the integral sign
and noting that
$V^{(n-1)}\in Y$,
we have
$$
	(1+t)^k|\partial_xv_1^{(n)}(t,x)|\le C,\quad
	(1+t)^k|\partial_xv_2^{(n)}(t,x)|\le C,\quad
	(1+t)^{k-1}|\partial_xh^{(n)}(t,x)|\le C
$$
with some constant
$C>0$.
Therefore we have
$V^{(n)}\in K_{\theta}$.

Next, we prove that
$\{ V^{(n)} \}_{n=0}^{\infty}$
is a Cauchy sequence.
By a straightforward calculation,
we can estimate
\begin{align*}
	&\sum_{\alpha=0,1}\sum_{j=1,2}
		|\partial_x^{\alpha}v_j^{(n)}(t,x)-\partial_x^{\alpha}v_2^{(n)}(t,x)|\\
	&\quad\le
		\int_t^{\infty}\frac{\mu}{1+s}
		\sum_{\alpha=0,1}\sum_{j=1,2}
		\|\partial_x^{\alpha}v_j^{(n-1)}-\partial_x^{\alpha}v_j^{(n-2)}\|_{\infty}ds\\
	&\qquad+\delta C\int_t^{\infty}\frac{1}{(1+s)^k}
		\sum_{\alpha=0,1}\sum_{j=1,2}
		\|\partial_x^{\alpha}v_j^{(n-1)}-\partial_x^{\alpha}v_j^{(n-2)}\|_{\infty}ds\\
	&\qquad+\delta C\int_t^{\infty}\frac{1}{(1+s)^{k+1}}
		\sum_{\alpha=0,1}
		\|\partial_x^{\alpha}h^{(n-1)}-\partial_x^{\alpha}h^{(n-2)}\|_{\infty}ds.
\end{align*}
Since $k>1$, this implies
\begin{equation*}
%\label{diffVn}
	\lambda(1+t)^k\sum_{\alpha=0,1}\sum_{j=1,2}
		|\partial_x^{\alpha}v_j^{(n)}(t,x)-\partial_x^{\alpha}v_j^{(n-1)}(t,x)|
		\le\left(\frac{\mu}{k}+\delta \lambda C\right)\|V^{(n-1)}-V^{(n-2)}\|_Y.
\end{equation*}
Using this, we can estimate the difference of $h^{(n)}$ and $h^{(n-1)}$ as
\begin{equation*}
%\label{diffhn}
	(1+t)^{k-1}\sum_{\alpha=0,1}|\partial_x^{\alpha}h^{(n)}(t,x)-\partial_x^{\alpha}h^{(n-1)}(t,x)|
	\le \frac{1}{2\lambda(k-1)}\left(\frac{\mu}{k}+\delta \lambda C\right)
	\|V^{(n)}-V^{(n-1)}\|_Y.
\end{equation*}
Adding the two inequalities above, we obtain
$$
	\|V^{(n)}-V^{(n-1)}\|_Y\le \left(1+\frac{1}{2\lambda(k-1)}\right)
		\left(\frac{\mu}{k}+\delta \lambda C\right)\|V^{(n-1)}-V^{(n-2)}\|_Y.
$$
Thus, by taking $\lambda$ sufficiently large and then
$\delta$ sufficiently small,
we obtain
$$
	\|V^{(n)}-V^{(n-1)}\|_Y\le r\|V^{(n-1)}-V^{(n-2)}\|_Y
$$
with some
$0<r<1$.
This completes the proof.
\end{proof}

\begin{proof}[Proof of Theorem \ref{th3}]%%%
First, we note that we can prove Lemma \ref{lem3}
by the same argument as the proof of Lemma \ref{lem1}.
Thus, we find a solution $g$ of \eqref{eqg} satisfying
\begin{equation}
\label{esg3}
	g\sim (1+t)^{\mu},\quad
	\|g_t(t)\|_{\infty}\lesssim (1+t)^{\mu-1},\quad
	\|g_x(t)\|_{\infty}\lesssim (1+t)^{\mu-k}.
\end{equation}
In what follows,
we use the same notation as in Section 3.
Suppose that
$u$
is a global solution.
Using the estimates \eqref{esg3},
we can obtain
\begin{align*}
	K_1&\lesssim \tau^{-2+(\mu+1)/q}R^{1/q}I_{\tau,R}^{1/p},\\
	K_2&\lesssim \tau^{(\mu+1)/q}R^{-2+1/q}I_{\tau,R}^{1/p},\\
	K_3&\lesssim G(\tau)R^{-1+1/q}I_{\tau,R}^{1/p},\\
	K_4&\lesssim \tau^{-2+(\mu+1)/q}R^{1/q}I_{\tau,R}^{1/p},
\end{align*}
where
$$
	G(\tau)=\left\{\begin{array}{ll}
		1&(\mu-kq<-1),\\
		(\log\tau)^{1/q}&(\mu-kq=-1),\\
		\tau^{-k+(\mu+1)/q}&(\mu-kq>-1).
	\end{array}\right.
$$
In this case we put
\begin{equation}
\label{eqD3}
	D(\tau,R):=\tau^{-2+(\mu+1)/q}R^{1/q}+
		\tau^{(\mu+1)/q}R^{-2+1/q}+G(\tau)R^{-1+1/q}.
\end{equation}
We note that the powers of each terms of
$D(\tau, \tau)$
do not exceed $0$ if
$1<p\le 1+2/\mu$.
Thus, by putting $R=\tau$
and the same argument as Sections 3 and 4,
we can lead a contradiction and complete the proof.
\end{proof}

\section{Estimates of Lifespan}
We can also give an upper estimate of the lifespan of the soluiton.
In this section, we follow the argument in \cite{IW} (see also \cite{Ku}).
We consider the initial value problem \eqref{eq11}
with the initial data
$(u,u_t)(0,x)=\varepsilon (u_0,u_1)(x)$
instead of
$(u_0,u_1)(x)$,
where
$\varepsilon$
denotes a positive small parameter.
For the sake of simplicity, we treat only the case that
$a$
satisfies \eqref{asa1}.
We define the lifespan of the solution by
$$
	T_{\varepsilon}:=\sup\{ T\in (0,\infty] \mid
	\text{ there is a unique solution}\  u\in X(T) \}.
$$
By Proposition \ref{prop1},
if
$(u_0,u_1)\in H^1(\mathbf{R})\times L^2(\mathbf{R})$,
then
$T_{\varepsilon}>0$
for any
$\varepsilon>0$.
We have an upper bound of
$T_{\varepsilon}$
as follows.
\begin{proposition}\label{prop61}%%%%
Let
$1<p<\infty$
and let
$(u_0,u_1)\in H^1(\mathbf{R})\times L^2(\mathbf{R})$
satisfy \eqref{eqini}
with
$g$
defined by \eqref{eqgh}
and
$h$
in Lemma \ref{lem1}.
Then
$T_{\varepsilon}$
is estimated as
\begin{equation}
\label{esTep}
	T_{\varepsilon}\le C\varepsilon^{-1/\kappa}
\end{equation}
with some constant
$C>0$
and
$\kappa=\frac{1}{p-1}(1+1/p)$.
\end{proposition}
\begin{proof}%%%
Again we use the same notation as in Section 3.
Here we note that
if $T_{\varepsilon}\le \tau_0$,
then the estimate \eqref{esTep} is obvious.
Thus, we may assume that
$T_{\varepsilon}> \tau_0$.

Now we use a fact that the inequality
$$
	dc^b-c\le (1-b)b^{b/(1-b)}d^{1/(1-b)}
$$
holds for all $d>0, 0<b<1, c\ge 0$.
We can immediately prove it by considering
the maximal value of the function
$f(c)=dc^b-c$.
From this and \eqref{eq34}, we obtain
\begin{equation}
\label{eq35}
	J_R\lesssim D(\tau, R)^q.
\end{equation}
On the other hand, by the assumption on the data,
there exist $C>0$ and $R_0$ such that
$J_R\ge C\varepsilon$ holds for all $R>R_0$.
Consequently, we have
\begin{equation}
\label{eq36}
	\varepsilon\lesssim D(\tau,R)^q
\end{equation}
for all
$\tau\in (\tau_0, T_{\varepsilon}), R\in (R_0,\infty)$.
We put $R=\tau^{\alpha}$ with $\alpha>0$
and
$\tau_0:=\max\{ 1, R_0^{1/\alpha} \}$.
Then we obtain
\begin{equation}
\label{eq37}
	D(\tau,\tau^{\alpha})\le
	\tau^{\max\{ -1-1/p+(1-1/p)\alpha, 1-1/p+(-1-1/p)\alpha, -\alpha/p \}}.
\end{equation}
Now we take
$\alpha=1+1/p$,
which minimizes the power of $\tau$ in \eqref{eq37}.
From \eqref{eq36} we see that
$$
	\varepsilon\lesssim D(\tau, \tau^{1+1/p})^q
	\lesssim \tau^{-\frac{1}{p-1}(1+1/p)}=\tau^{-\kappa}.
$$
Therefore, we have
$$
	\tau\le C\varepsilon^{-1/\kappa}.
$$
Since $\tau$ is arbitrarily in $(\tau_0, T_{\varepsilon})$,
it follows that
$$
	T_{\varepsilon}\le C\varepsilon^{-1/\kappa},
$$
which completes the proof.
\end{proof}

\section*{acknowledgement}
The author is deeply grateful to Professors
Akitaka Matsumura,
Kenji Nishihara
and
Tatsuo Nishitani
for the discussion at the starting point of this study.


\begin{thebibliography}{9}
\bibitem{D}\textsc{M. D'Abbicco},
{\em The threshold between effective and noneffective damping for semilinear waves},
arXiv:1211.0731v2.

\bibitem{DL}\textsc{M. D'abbicco, S. Lucente},
{\em A modified test function method for damped wave equations},
arXiv:1211.0453v1.

\bibitem{DLR}\textsc{M. D'Abbicco, S. Lucente, M. Reissig},
{\em Semi-Linear wave equations with effective damping},
Chin. Ann. Math., Ser. B, {\bf 34} (2013), 345-380.

\bibitem{F}\textsc{H. Fujita},
{\em On the blowing up of solutions of the Cauchy problem for $u_t=\Delta u+u^{1+\alpha}$},
J. Fac. Sci. Univ. Tokyo Sect. I, {\bf 13} (1966), 109-124.

\bibitem{HKN04}\textsc{N. Hayashi, E. I. Kaikina, P. I. Naumkin},
{\em Damped wave equation with super critical nonlinearities},
Differential Integral Equations, {\bf 17} (2004), 637-652.

\bibitem{HO}\textsc{T. Hosono, T. Ogawa},
{\em Large time behavior and $L^p$-$L^q$ estimate of solutions of 2-dimensional nonlinear damped wave equations},
J. Differential Equations, {\bf 203} (2004) 82-118.

\bibitem{IW}\textsc{M. Ikeda, Y. Wakasugi},
{\em A note on the lifespan of solutions to the semilinear damped wave equation},
Proc. Amer. Math. Soc. (to appear)

\bibitem{IT}\textsc{R. Ikehata, K. Tanizawa},
{\em Global existence of solutions for semilinear damped wave equations
in $\mathbf{R}^N$ with noncompactly supported initial data},
Nonliear Anal., {\bf 61} (2005), 1189-1208.

\bibitem{ITY}\textsc{R. Ikehata, G. Todorova, B. Yordanov},
{\em Critical exponent for semilinear wave equations with space-dependent potential},
Funkcialaj Ekvacioj, {\bf 52} (2009), 411-435.

\bibitem{ITY2}\textsc{R. Ikehata, G. Todorova, B. Yordanov},
{\em Optimal decay rate of the energy for wave equations with critical potential},
J. Math. Soc. Japan, {\bf 65} (2013), 183-236.

\bibitem{K}\textsc{T. Kato},
{\em Blow-up of solutions of some nonlinear hyperbolic equations},
Comm. Pure Appl. Math., {\bf 33} (1980), 501-505.

\bibitem{KK}\textsc{J. S. Kenigson, J. J. Kenigson},
{\em Energy decay estimates for the dissipative wave equation with space-time dependent potential},
Math. Meth. Appl. Sci., {\bf 34} (2011), 48-62.

\bibitem{Kh}\textsc{M. Khader},
{\em Global existence for the dissipative wave equations with space-time dependent potential},
Nonlinear Anal., {\bf 81} (2013), 87-100.

\bibitem{Ku}\textsc{H. J. Kuiper},
{\em Life span of nonegative solutions to certain qusilinear parabolic Cauchy problems},
Electron. J. Differential Equations, {\bf 2003} (2003), 1-11.

\bibitem{LNZ}\textsc{J. Lin, K. Nishihara, J. Zhai},
{\em Critical exponent for the semilinear wave equation with time-dependent damping},
Discrete Contin. Dyn. Syst., {\bf 32} (2012), 4307-4320.

\bibitem{MN}\textsc{P. Marcati, K. Nishihara},
{\em The $L^p-L^q$ estimates of solutions to one-dimensional damped wave equations and their application to the compressible flow through porous media},
J. Differential Equations, {\bf 191} (2003), 445-469.


\bibitem{M}\textsc{A. Matsumura},
{\em On the asymptotic behavior of solutions of semi-linear wave equations},
Publ. Res. Inst. Math. Sci., {\bf 12} (1976), 169-189.

\bibitem{Mo}\textsc{K. Mochizuki},
{\em Scattering theory for wave equations with dissipative terms},
Publ. Res. Inst. Math. Sci., {\bf 12} (1976), 383-390.

\bibitem{Na}\textsc{T. Narazaki},
{\em $L^p$-$L^q$ estimates for damped wave equations and their applications to semi-linear problem},
J. Math. Soc. Japan, {\bf 56} (2004), 585-626.

\bibitem{N1}\textsc{K. Nishihara},
{\em $L^p-L^q$ estimates of solutions to the damped wave equation in 3-dimensional space and their application},
Math. Z., {\bf 244} (2003), 631-649.

\bibitem{TY1}\textsc{G. Todorova, B. Yordanov},
{\em Critical exponent for a nonlinear wave equation with damping},
J. Differential Equations, {\bf 174} (2001), 464-489.

\bibitem{TY3}\textsc{G. Todorova, B. Yordanov},
{\em Weighted $L^2$-estimates for dissipative wave equations with
variable coefficients},
J. Differential Equations, {\bf 246} (2009), 4497-4518.

\bibitem{Wa1}\textsc{Y. Wakasugi},
{\em Small data global existence for the semilinear wave equation with space-time dependent damping},
J. Math. Anal. Appl., {\bf 393} (2012) 66-79.

\bibitem{Wa2}\textsc{Y. Wakasugi},
{\em Critical exponent for the semilinear wave equation with scale invariant damping},
Trends in Mathematics, Birkh\"auser, Basel (to appear).

\bibitem{Wi04}\textsc{J. Wirth},
{\em Solution representations for a wave equation with weak dissipation},
Math. Meth. Appl. Sci., {\bf 27} (2004), 101-124.

\bibitem{W1}\textsc{J. Wirth},
{\em Wave equations with time-dependent dissipation I. Non-effective dissipation},
J. Differential Equations, {\bf 222} (2006), 487-514.

\bibitem{W2}\textsc{J. Wirth},
{\em Wave equations with time-dependent dissipation II. Effective dissipation},
J. Differential Equations {\bf 232} (2007), 74-103.

\bibitem{HM}\textsc{H. Yang, A. Milani},
{\em On the diffusion phenomenon of quasilinear hyperbolic waves},
Bull. Sci. Math., {\bf124} (2000), 415-433.

\bibitem{Z}\textsc{Qi S. Zhang},
{\em A blow-up result for a nonlinear wave equation with damping: The critical case},
C. R. Acad. Sci. Paris S\'{e}r. I Math., {\bf 333} (2001), 109-114.

\end{thebibliography}
\end{document}